\documentclass[12pt]{article}

\usepackage{amsmath,amssymb,amstext,amsthm}

\theoremstyle{plain}
\newtheorem{theorem}{Theorem}
\newtheorem{lemma}{Lemma}

\newtheorem{corollary}[lemma]{Corollary}

\theoremstyle{definition}

\theoremstyle{remark}

\title{On the Maximum Density of \\ Graphs with Good Edge-Labellings}
\author{Abbas Mehrabian\thanks{Department of Combinatorics and Optimization, University of Waterloo, Waterloo, ON, Canada} \and Dieter Mitsche\thanks{Universit\'{e} de Nice Sophia-Antipolis, Laboratoire J-A Dieudonn\'{e}, Parc Valrose, 06108 Nice cedex 02} \and Pawe\l{} Pra\l{}at\thanks{Department of Mathematics, Ryerson University, Toronto, ON, Canada}}
\date{}
\begin{document}

\maketitle

\begin{abstract}
A \emph{good edge-labelling} of a simple, 
finite graph is a labelling of its edges with real numbers such that, for every ordered pair of vertices $(u,v)$, there is at most one nondecreasing path from $u$ to $v$.
In this paper we prove that any graph on $n$ vertices that admits a good edge-labelling has at most $n \log_2(n)/2$ edges,
and that this bound is tight for infinitely many values of $n$.
Thus we significantly improve on the previously best known bounds.
The main tool of the proof is a combinatorial lemma which might be of independent interest.
For every $n$ we also construct an $n$-vertex graph that admits a good edge-labelling
and has $n\log_2(n)/2 - O(n)$ edges.
\end{abstract}

\section{Introduction}

Let $G$ be a finite, simple graph.
A \emph{good edge-labelling} of $G$ is a labelling of its edges with real numbers such that, for any ordered pair of vertices $(u,v)$, there is at most one nondecreasing path from $u$ to $v$. This notion was introduced in~\cite{origin} to solve wavelength assignment problems for specific categories of graphs. We say $G$ is \emph{good} if it admits a good edge-labelling.

Let $f(n)$ be the maximum number of edges of a good graph on $n$ vertices. Ara{\'u}jo, Cohen, Giroire, and Havet~\cite{gn} initiated the study of this function. 
They observed that hypercube graphs are good and that any graph containing $K_3$ or $K_{2,3}$ is not good.
From these observations they concluded that if $n$ is a power of two, then
$$
f(n) \ge \frac{n}{2} \log_2(n) \:,
$$
and that for all $n$,
$$
f(n) \leq \frac{n\sqrt n}{\sqrt 2} + O\left(n^{4/3}\right) \:.
$$
The first author of this paper proved that any good graph whose maximum degree is within a constant factor of its average degree (in particular, any good regular graph) has at most $n^{1+o(1)}$ edges---see~\cite{abbas} for more details. 

Before we state the main result of this paper, we need one more definition. 
Let $b(n)$ be the function that counts the total number of $1$'s in the binary expansions of all integers from $0$ up to $n-1$. This function was studied in~\cite{binary}. 
Our main result is the following theorem.

\begin{theorem}\label{thm:main}
For all positive integers $n$,
$$
\frac{n}{2} \log_2 \left( \frac{3n}{4} \right) \leq b(n) \leq f(n) \leq \frac{n}{2} \: \log_2(n) \:.
$$
\end{theorem}

It follows that the asymptotic value of $f(n)$ is 
$n \log_2(n) / 2 - O(n)$. 
Note that Theorem~\ref{thm:main} implies that \emph{any} good graph on $n$ vertices has at most $n\log_2(n)/2$ edges, 
significantly improving the previously known upper bounds.
Moreover, this bound is tight if $n$ is a power of two.
We also give an explicit construction of a good graph with $n$ vertices and $b(n)$ edges for every $n$.

\section{The Proofs}

This section is devoted to prove the main result, Theorem~\ref{thm:main}.

\subsection{The upper bound}

For a graph $G$, an edge-labelling $\phi:E(G) \rightarrow \mathbb{R}$, and an integer $t \geq 0$, a \emph{nice $t$-walk from $v_0$ to $v_t$} is a sequence $v_0v_1\dots v_t$ of vertices such that $v_{i-1}v_i$ is an edge for $1\leq i\leq t$, and $v_{i-1}\neq v_{i+1}$ and $\phi(v_{i-1} v_{i}) \leq \phi (v_{i}v_{i+1})$ for $1 \leq i \leq t-1$. We call $v_t$ the \emph{last vertex} of the walk. When $t$ does not play a role, we simply refer to a \emph{nice walk}.  The existence of a self-intersecting nice walk implies that the edge-labelling is not good: let $v_0 v_1 \dots v_t$ be a shortest such walk with $v_0 = v_t$. Then there are two nondecreasing paths $v_0 v_1 \dots v_{t-1}$ and $v_0 v_{t-1}$ from $v_0$ to $v_{t-1}$. Thus if for some pair of distinct vertices $(u,v)$ there are two nice walks from $u$ to $v$, then the labelling is not good. Also, if for some vertex $v$, there is a nice $t$-walk from $v$ to $v$ with $t>0$, then the labelling is not good. Consequently, if the total number of nice walks is larger than $2\binom{n}{2}+n=n^2$, then the labelling is not good.

The following lemma will be very useful.

\begin{lemma}\label{lem:Edgelabel}
Let $G$ and $H$ be graphs with good edge-labellings on disjoint vertex sets. Then if we add a matching between the vertices of $G$ and $H$ (i.e., add a set of edges, such that each added edge has an endpoint in either of $V(G)$ and $V(H)$, and every vertex in $V(G) \cup V(H)$ is incident to at most one added edge), then the resulting graph is good.
\end{lemma}
\begin{proof}
Consider a good edge-labelling  of $G$ and $H$, and let $M$ be a number greater than all existing labels. Then label the matching edges with $M, M+1, M+2$, etc. It is not hard to verify that the resulting edge-labelling is still good.
\end{proof}

\begin{corollary}
\label{cor:recursive}
We have $f(1) = 0$ and for all $n > 1$,
$$f(n) \geq \max \Big\{ f(n_1) + f(n_2) + \min \{ n_1, n_2\}: 1 \leq n_1, 1\leq n_2, n_1 + n_2 = n \Big\} \:.$$
\end{corollary}

The proof of the upper bound in Theorem~\ref{thm:main} relies on the analysis of a one-player game, which is defined next. 
The player, who will be called Alice henceforth, starts with $n$ sheets of paper, on each of which a positive integer is written. 
In every step, Alice does an operation as follows. She chooses any two sheets. Assume that the numbers written on them are $a$ and $b$. She erases these numbers, and writes $a+b$ on both sheets. We denote the move by a pair $(a,b)$. The \emph{configuration} of the game is a multiset of size $n$, containing the numbers written on the sheets, in which the multiplicity of number $x$ equals the number of sheets on which $x$ is written. Notice that there may be multiple pairs of sheets with numbers $a$ and $b$ written on them, and we treat choosing any such pair as the same move, since they all result in the same configuration (that is, all moves are isomorphic). Note also that the moves $(a,b)$ and $(b,a)$ have the same effect.

Clearly, after playing the move $(a,b)$, the sum of the numbers is increased by $a+b$. The aim of the game is to keep the sum of the numbers smaller than a certain threshold. Let $S$ be  the \emph{starting configuration} of the game, namely, a multiset of size $n$ containing the numbers initially written on the sheets, and let $k\geq0$ be an integer. We denote by $opt(S,k)$ the smallest sum Alice can get after doing $k$ operations. An intuitively good-looking strategy is the following: in each step, choose two sheets with the smallest numbers. We call this the \emph{greedy} strategy, and show that it is indeed an optimal strategy. 
Specifically, we prove the following theorem, which may be of independent interest.

\begin{theorem} \label{thm:greedy}
For any starting configuration $S$ and any nonnegative integer $k$, if Alice plays the greedy strategy, then the sum of the numbers after $k$ moves equals $opt(S,k)$.
\end{theorem}

Before proving Theorem~\ref{thm:greedy}, we show how this implies our upper bound.

\begin{proof}
[Proof of the upper bound of Theorem~\ref{thm:main}]
\label{thm:upper_bd}
Let $G$ be a graph with $n$ vertices and $m>{n} \log_2(n) / 2$ edges.
We need to show that $G$ does not have a good edge-labelling.
Consider an arbitrary edge-labelling $\phi:E(G) \rightarrow \mathbb{R}$. 
Enumerate the edges of $G$ as $e_1,e_2,\dots,e_m$ such that
$$\phi(e_1) \leq \phi(e_2) \leq \dots \leq \phi(e_m) \:.$$
We may assume that the inequalities are strict.
Indeed, if some label $L$ appears $p > 1$ times, we can assign the labels $L,L+1,\ldots,L+(p-1)$ to the edges originally labelled $L$, 
and increase  by $p$ the label of edges with original label larger than $L$. 
It is easy to see that the modified edge-labelling is still good, and by repeatedly applying this operation all ties are broken.

Let us denote by $G_i$ the subgraph of $G$ induced by $\{e_1,e_2,\dots,e_i\}$.
For each vertex $v$ and $0\leq i \leq m$, let $a_v^{(i)}$ be the number of nice walks with last vertex $v$ in $G_i$. 
Clearly, $a_v^{(0)} = 1$ for all vertices $v$.
Suppose the graph is initially empty and we add the edges $e_1,e_2,\dots,e_m$ one by one, in this order. Fix an $i$ with $1\leq i \leq m$. Let $u$ and $v$ be the endpoints of $e_i$. 
After adding the edge $e_i$, for any $t$, any nice $t$-walk with last vertex $u$ (respectively, $v$)  
in $G_{i-1}$ can be extended via $e_i$ to a nice $(t+1)$-walk with last vertex $v$ (respectively, $u$) in $G_i$. 
So, we have $a_u^{(i)} = a_v^{(i)} = a_u^{(i-1)} + a_v^{(i-1)}$ and $a_w^{(i)} = a_w^{(i-1)}$ for $w\notin\{u,v\}$.

Thus we are in the same setting as the one-player game described before, with starting configuration $S = \{1,1,\dots,1\}$, so we have
$$
\sum_{v\in V(G)} a_v^{(m)} \geq opt(S, m)\:.
$$
Hence, in order to prove that $\phi$ is not a good edge-labelling, it is sufficient to show that $opt(S,m) > n^2$.

Let $m_0$ be the largest number for which $opt(S,m_0) \leq n^2$,
and let $\alpha = \lfloor \log_2(n) \rfloor$.
First, assume that $n$ is even.
By Theorem~\ref{thm:greedy}, we may assume that Alice plays according to the greedy strategy.
The smallest number on the sheets is initially 1, and is doubled after every $n/2$ moves.
Hence after $\alpha n/2$ moves, the smallest number becomes $2^{\alpha}$,
so the sum of the numbers would be 
$2^{\alpha} n$.
In every subsequent move, the sum is increased by 
$2^{\alpha+1}$,
so Alice can play at most $(n^2 - 2^{\alpha} n) / 2^{\alpha+1}$ more moves before the sum of the numbers becomes greater than $n^2$.
Consequently,
$$m_0 \leq \alpha \: \frac{n}{2} + \frac{n (n-2^{\alpha})}{2^{\alpha+1}} \:.$$

Now, define $h(x) := \log_2(x) - x + 1$.
Then $h$ is concave {in $[1,2]$} and $h(1)=h(2)=0$, which implies that $h(x) \geq 0$ for all $x \in [1,2]$.
In particular, for $x_0 = n / 2^{\alpha}$, we have
$$\frac{n-2^{\alpha}}{2^{\alpha}} = x_0-1 \leq \log_2(x_0) = \log_2\left(\frac{n}{2^{\alpha}}\right) = \log_2(n) - \alpha \:.$$
Therefore,
$$m_0 \leq \frac{n}{2}\: \alpha + \frac{n}{2} \: \frac{n - 2^{\alpha}}{2^{\alpha}} \leq \frac{n}{2}\: \log_2(n) < m\:,$$
which completes the proof.

Finally, assume that $n$ is odd.
Since $2n$ is even, we have
$$f(2n) \leq n \log_2(2n) = n\log_2(n) + n \:.$$
On the other hand, by Corollary~\ref{cor:recursive},
$$f(2n) \geq 2f(n) + n \:.$$
Combining these inequalities gives 
$$f(n) \leq \frac{n}{2} \log_2(n) \:,$$
completing the proof of the lemma.
\end{proof}

The rest of this section is devoted to prove Theorem~\ref{thm:greedy}. Let $S = \{s_1,s_2,\dots,s_n\}$ be the starting configuration of the game. Note that after each step, the sum of the  numbers is of the form $\sum_{i=1}^{n} c_i s_i$ for some positive integers $\{c_i\}_{i=1}^{n}$. The sequence $\{c_i\}_{i=1}^{n}$ depends only on the sheets Alice chooses in each step, and does not depend on $\{s_i\}_{i=1}^{n}$. We say that $(c_1,c_2,\dots,c_n)$ is \emph{$k$-feasible} if {after $k$ steps} Alice can actually get a sum of the form $\sum_{i=1}^{n} c_i s_i$. For example, $(1,1,\dots,1)$ is the only $0$-feasible $n$-tuple, and there exist $\binom{n}{2}$ different 1-feasible $n$-tuples, one of them being $(2,2,1,1,\dots,1)$. Notice that for any permutation $\pi$ of $\{1,2,\dots,n\}$, if $(c_1,c_2,\dots,c_n)$ is $k$-feasible, then so is $(c_{\pi(1)},c_{\pi(2)},\dots,c_{\pi(n)})$, since Alice can first permute the sheets according to the permutation $\pi$, and then apply the same strategy as before.

For multisets $S$ and $T$ of size $n$, we write $S \leq T$ if for all $k \geq 0$ we have $opt(S,k) \leq opt(T,k)$. Note that if we can arrange the elements of $S$ and $T$ as $S=\{s_1,s_2,\dots,s_n\}$ and $T=\{t_1,t_2,\dots,t_n\}$ such that $s_i \leq t_i$ holds for all $1\leq i \leq n$, then $S \leq T$.

First, we make two useful observations.

\begin{lemma}\label{lem_3best}
Let $k\geq 2 $ and assume that the starting configuration is
$$
S = \{a, b, c, x_1, x_2, \dots, x_m \},
$$
where $a\leq b \leq c$. Also suppose that either there is an optimal $k$-step strategy in which Alice plays $(b,c)$ and $(a,b+c)$ in the first and second steps, or there is an optimal $k$-step strategy in which Alice plays $(a,c)$ and $(b,a+c)$ in the first and second steps. Then, there exists an optimal $k$-strategy in which Alice plays $(a,b)$ and $(c,a+b)$ in the first and second steps.
\end{lemma}
\begin{proof}
In the first case, the configuration after the second step is
$$
U = \{b+c, a+b+c, a + b + c, x_1,x_2,\dots,x_m\}.
$$
In the second case, the configuration after the second step is
$$
V = \{a+c, a+b+c, a+b+c, x_1,x_2,\dots,x_m\}.
$$
In the third case, the configuration after the second step is
$$
T = \{a+b, a+b+c, a+b+c, x_1,x_2,\dots, x_m\}.
$$
Since $a\leq b\leq c$, we have $T \leq U$ and $T \leq V$, and therefore $opt(T,k-2) \leq opt(U,k-2)$ and $opt(T,k-2) \leq opt(V,k-2)$. Hence, the strategy starting with moves $(a,b)$ and $(c,a+b)$ is also an optimal $k$-step strategy.
\end{proof}

\bigskip

\begin{lemma}
\label{lem:weights}
Let $x_1,x_2,\dots,x_m,a,b,c,d$ be positive integers with $a \leq b \leq c \leq d$.
Define $S$, $T$, and $U$ as follows:
 \begin{align*}
S & := \{x_1,x_2, \dots, x_m, a + b, a + b, c + d, c + d\}\:;  \\
T & := \{x_1, x_2,\dots, x_m, a + c, a + c, b + d, b + d\}\:; \\
U & := \{x_1, x_2, \dots, x_m, a + d, a + d, b + c, b + c\} \:.
\end{align*}
Then, we have $S \leq T$ and $S \leq U$.
\end{lemma}

\begin{proof}
Let $k \geq 0$.
We first show that $opt(S,k) \leq opt(T,k)$.
Let $w_1,w_2,\dots,w_n$ be such that
$$opt(T,k) = \sum_{i=1}^m w_i x_i + w_{m+1} (a+c) + w_{m+2} (a+c) + w_{m+3} (b+d) + w_{m+4} (b+d) \:.$$
Since
$(w_1,w_2,\dots,w_m, w_{m+1}, w_{m+2}, w_{m+3}, w_{m+4})$
and
$(w_1,w_2,\dots,w_m, w_{m+3}, w_{m+4}, w_{m+1}, w_{m+2})$
are $k$-feasible, we have
\begin{align*}
opt(S,k) & \leq \sum_{i=1}^m w_i x_i + w_{m+1} (a+b) + w_{m+2} (a+b) + w_{m+3} (c+d) + w_{m+4} (c+d)
\\ & = opt(T,k) + (w_{m+1}+w_{m+2} - w_{m+3} - w_{m+4}) (b-c) \:,
\end{align*}
and
\begin{align*}
opt(S,k) & \leq \sum_{i=1}^m w_i x_i + w_{m+3} (a+b) + w_{m+4} (a+b) + w_{m+1} (c+d) + w_{m+2} (c+d) \\
& = opt(T,k) + (w_{m+1}+w_{m+2}-w_{m+3}-w_{m+4}) (d-a) \:.
\end{align*}
Now, if $ w_{m+1}+w_{m+2} - w_{m+3} - w_{m+4}  \geq 0$,
then the first inequality gives $opt(S,k) \leq opt(T,k)$,
and otherwise, the second inequality gives $opt(S,k) \leq opt(T,k)$.

Similarly, we show that $opt(S,k) \leq opt(U,k)$.
Let $w_1,w_2,\dots,w_n$ be such that
$$opt(U,k) = \sum_{i=1}^m w_i x_i + w_{m+1} (a+d) + w_{m+2} (a+d) + w_{m+3} (b+c) + w_{m+4} (b+c) \:.$$
As before, let us notice that both
$(w_1,w_2,\dots,w_m, w_{m+1}, w_{m+2}, w_{m+3}, w_{m+4})$
as well as
$(w_1,w_2,\dots,w_m, w_{m+3}, w_{m+4}, w_{m+1}, w_{m+2})$
are $k$-feasible, hence we have
\begin{align*}
opt(S,k) & \leq \sum_{i=1}^m w_i x_i + w_{m+1} (a+b) + w_{m+2} (a+b) + w_{m+3} (c+d) + w_{m+4} (c+d)
\\ & = opt(U,k) + (w_{m+1}+w_{m+2} - w_{m+3} - w_{m+4}) (b-d) \:,
\end{align*}
and
\begin{align*}
opt(S,k) & \leq \sum_{i=1}^m w_i x_i + w_{m+3} (a+b) + w_{m+4} (a+b) + w_{m+1} (c+d) + w_{m+2} (c+d) \\
& = opt(U,k) + (w_{m+1}+w_{m+2}-w_{m+3}-w_{m+4}) (c-a) \:.
\end{align*}
Now, if $ w_{m+1}+w_{m+2} - w_{m+3} - w_{m+4}  \geq 0$,
then the first inequality gives $opt(S,k) \leq opt(U,k)$,
and otherwise, the second inequality gives $opt(S,k) \leq opt(U,k)$.
\end{proof}

Lemma~\ref{lem:weights} immediately implies the following corollary.
\begin{corollary}\label{cor_opt}
Let $k \geq 2$ and let $a\leq b \leq c \leq d$ be four numbers in the starting configuration. Suppose that either  there is an optimal $k$-step strategy in which Alice plays $(a,c)$ and $(b,d)$ in the first two steps, or there is an optimal $k$-step strategy in which Alice plays $(a,d)$ and $(b,c)$ in the first two steps. Then there exists an optimal $k$-step strategy in which Alice plays $(a,b)$ and $(c,d)$ in the first and second steps.
\end{corollary}

The following lemma finishes the proof of Theorem~\ref{thm:greedy}.
\begin{lemma}
Let $k\geq 0$, and let $S = \{s_1,s_2,\dots, s_n\}$ be the starting configuration, arranged such that $s_1\leq s_2 \leq \dots \leq s_n$.  Then there exists an optimal $k$-step strategy with the first move $(s_1,s_2)$.
\end{lemma}

\begin{proof}
We prove the lemma by induction on $k$. The induction base is obvious for $k = 0$ and easy for $k=1$, so assume that $k\geq 2$. Consider an optimal $k$-step strategy. Assume that the first move is $(s_i,s_j)$, where $1\leq i < j \leq n$. If $i=1$ and $j=2$, then we are done. Otherwise, there are 5 cases to consider:
\begin{description}
\item [$\{i,j\}=\{1,3\}$.]
There are two subcases:

\begin{enumerate}
\item If $s_1 + s_3 \leq s_4$, then by the induction hypothesis, for the configuration that arose after the first step, there is an optimal $(k-1)$-step strategy in which Alice's first move is $(s_2,s_1+s_3)$. That is, there is an optimal $k$-step strategy for the initial configuration in which Alice plays $(s_1,s_3)$ and $(s_2,s_1+s_3)$ in the first two steps. Thus, by Lemma~\ref{lem_3best} there also exists an optimal $k$-step strategy (for the initial configuration) with first move $(s_1,s_2)$.

\item If $s_1 + s_3 > s_4$, then by the induction hypothesis, there is an optimal strategy with second move $(s_2,s_4)$. Then, by Lemma~\ref{lem:weights},  there exists an optimal  strategy with first move $(s_1,s_2)$.
\end{enumerate}

\item [$\{i,j\}=\{2,3\}$.]
As before, there are two subcases:
\begin{enumerate}
\item If $s_2 + s_3 \leq s_4$, then by the induction hypothesis, there is an optimal strategy with second move $(s_1,s_2+s_3)$.  As before, by Lemma~\ref{lem_3best},  there exists an optimal  strategy with  first move $(s_1,s_2)$.

\item If $s_2 + s_3 > s_4$, then by the induction hypothesis, there is an optimal strategy with second move $(s_1,s_4)$. Then, by Lemma~\ref{lem:weights}, there exists an optimal  strategy with first move $(s_1,s_2)$.
\end{enumerate}

\item [$i=1$ and $j>3$.] By the induction hypothesis, there is an optimal strategy with second move $(s_2,s_3)$, and thus by Lemma~\ref{lem:weights}, there exists an optimal strategy with first move $(s_1,s_2)$.

\item [$i=2$ and $j>3$.] By the induction hypothesis, there is an optimal strategy with second move $(s_1,s_3)$, and again, Lemma~\ref{lem:weights} implies that  there exists an optimal strategy with  first move $(s_1,s_2)$.

\item [$i>2$ and $j>3$.] By the induction hypothesis, there is an optimal strategy with second move $(s_1,s_2)$. Swapping the first and second moves gives an optimal strategy with first move $(s_1,s_2)$.\qedhere
\end{description}
\end{proof}

\subsection{The lower bound}

In this section we prove the lower bound in Theorem~\ref{thm:main}.
Recall that $b(n)$ is equal to the total number of $1$'s in the binary expansions of all integers from $0$ up to $n-1$.
It is known~\cite{binary} that $b(1)=0$ and $b(n)$ satisfies the recursive formula
$$b(n) = \max \{ b(n_1) + b(n_2) + \min \{ n_1, n_2\}: 1 \leq n_1, 1\leq n_2, n_1 + n_2 = n\} \:,$$
and the lower bound in Theorem~\ref{thm:main} follows by using induction and applying Corollary~\ref{cor:recursive}.
Moreover, McIlroy~\cite{binary} proved that $b(n) \geq {n} \log_2 \left( \frac{3}{4} n \right) / 2$. 

For every $n$ we also give an explicit construction of a good graph with $n$ vertices and $b(n)$ edges.
It is easy to see that $b(n)$ equals the number of edges
in the graph $G_n$ with vertex set $\{0,1,\dots,n-1\}$, and with vertices $i$ and $j$ being adjacent
if the binary expansions of $i$ and $j$ differ in exactly one digit.
This graph is an induced subgraph of the $\lceil \log_2(n) \rceil$-dimensional hypercube graph.
It can be shown by induction and Lemma~\ref{lem:Edgelabel} that the hypercube graph is good,
which implies that $G_n$ is also good
(since the restriction of a good edge-labelling for the supergraph
to the edges of the subgraph is a good edge-labelling for the subgraph).
Hence $G_n$ is a good graph with $n$ vertices and $b(n)$ edges.

\section{Concluding Remarks}

We proved that any $n$-vertex graph with a good edge-labelling has at most $n\log_2(n)/2$ edges,
and for every $n$ we constructed a good $n$-vertex graph with $n\log_2(n)/2 - O(n)$ edges.
Thus we proved 
$f(n) = n\log_2(n)/2 - O(n)$.
One can try to investigate the second order term of the function $f(n)$.
Perhaps it is the case that our construction is best possible; that is, in fact $f(n) = b(n)$?

It would be interesting to further investigate the connection between
having a good edge-labelling and other parameters of the graph;
in particular, the length of the shortest cycle (known as the girth) of the graph (see, e.g.,~\cite{conjecture}).
Ara{\'u}jo~et~al.~\cite{gn} proved that any planar graph with girth at least 6 has a good edge-labelling,
and asked whether 6 can be replaced with 5 in this result.
The first author~\cite{abbas} proved that any graph with maximum degree $\Delta$
and girth at least $40\Delta$ is good.
This does not seem to be tight, and improving the dependence on $\Delta$ is an interesting research direction.


\end{document}